\documentclass{amsart}

\usepackage{amsmath}
\usepackage{amssymb}
\usepackage[]{amsrefs}
\usepackage{xpatch}
\usepackage{kantlipsum} 
\calclayout

\xpatchcmd{\proof}{\itshape}{\prooflabelfont}{}{}
\newcommand{\prooflabelfont}{\bfseries}

\DefineSimpleKey{bib}{primaryclass}{}
\DefineSimpleKey{bib}{archiveprefix}{}

\BibSpec{arXiv}{%
  +{}{\PrintAuthors}{author}
  +{,}{ \textit}{title}
  +{}{ \parenthesize}{date}
  +{,}{ arXiv }{eprint}
  +{,}{ primary class }{primaryclass}
}

\usepackage{kantlipsum} 
\calclayout
\usepackage{extarrows}
\usepackage{amssymb}
\usepackage[utf8]{inputenc}
\newtheorem{theorem}{Theorem}[section]
\newtheorem{proposition}[theorem]{Proposition}
\newtheorem{lemma}[theorem]{Lemma}
\newtheorem{corollary}[theorem]{Corollary}
\theoremstyle{definition}

\theoremstyle{definition}
\newtheorem{definition}[theorem]{Definition}
\newtheorem{remark}[theorem]{Remark}

\numberwithin{equation}{section}
\usepackage{amsthm,amsmath,amssymb}
\usepackage[all,cmtip]{xy}
\usepackage{amsmath}
\usepackage{amssymb}
\usepackage{amsthm}
\usepackage[]{amsrefs}
\usepackage{hyperref}
\usepackage{xpatch}
\usepackage{amsfonts}
\usepackage{amssymb}
\usepackage[utf8]{inputenc}
\usepackage{amsthm}
\usepackage{stmaryrd}
\usepackage{csquotes}
\usepackage{extarrows}
\MakeOuterQuote{"}
\theoremstyle{definition}

\usepackage{enumitem}
\usepackage{tikz-cd}

\usepackage{blindtext}

\usepackage{url}

\DeclareMathOperator{\Ext}{Ext}

\begin{document}

\title[The index of a numerical semigroup ring]{The index of a numerical semigroup ring}

\author[Richard F. Bartels]{Richard F. Bartels}

\address{Department of Mathematics, Trinity College, 300 Summit St, Hartford, CT, 06106}

\email{rbartels@trincoll.edu}

\urladdr{https://sites.google.com/view/richard-bartels-math/home}



\title{Cohen-Macaulay approximations over generically Gorenstein rings}
\begin{abstract} Let $(R,\mathfrak{m})$ be a Cohen-Macaulay local ring with canonical module that is generically Gorenstein. In this paper, I prove isomorphisms relating the minimal MCM approximations and minimal FID hulls of modules constructed from a canonical ideal $\omega \subset R$, including \,$\omega/xR$, with $x \in \omega$ a nonzerodivisor,  $(\omega/xR)^{\vee}:=\text{Ext}^1_R(\omega/xR,\omega)$, $R/\omega^2$, and $\omega/\omega^2$. I also prove that if $R$ is not Gorenstein, then \, $\delta_{R}\left(\omega/xR \right)=\delta_{R}\left(\left(\omega/xR \right)^{\vee} \right)=0$ \,and $\gamma_{R}\left(\Omega^{1}_{R}\left(\omega/xR \right) \right)=\gamma_{R}\left(\Omega^{1}_{R}\left(\left(\omega/xR\right)^{\vee}\right) \right)=0$, where $\delta_R$ is Auslander's\, $\delta$-invariant and $\gamma_R$ is the dual $\gamma$-invariant. The results in this paper build on \cite[section 2]{Bartels26}.
\end{abstract}
\maketitle
\large{

\begin{center}
\section{Introduction}
\end{center} \text{} Let $(R,\mathfrak{m})$ be a Cohen-Macaulay local ring with canonical module $\omega$. In the following, I denote the direct sum of $n$ copies of a module $M$ by $M^{(n)}$. For every finitely-generated $R$-module $M$, there is an exact sequence 
\[
0 \longrightarrow Y \xlongrightarrow{\iota} X \longrightarrow M \longrightarrow 0
\] with $Y$ an $R$-module of finite injective dimension and $X$ a maximal Cohen-Macaulay (MCM) $R$-module, called an {\it{MCM approximation}} of $M$. If $Y$ and $X$ have no direct summand in common via $\iota$, then the MCM approximation is {\it{minimal}}, and denoted as follows.
\[
0 \longrightarrow Y_M \longrightarrow X_M \longrightarrow M \longrightarrow 0
\] Dually, there is an exact sequence of $R$-modules 
\[
0 \longrightarrow M \longrightarrow Y' \xlongrightarrow{\pi} X' \longrightarrow 0
\] with $Y'$ an $R$-module of finite injective dimension and $X'$ either an MCM $R$-module or zero, called a {\it{finite injective dimension hull}} (FID hull) of $M$. If $Y'$ and $X'$ have no direct summand in common via $\pi$, then the FID hull is minimal, and denoted as follows.
\[
0 \longrightarrow M \longrightarrow Y^M \longrightarrow X^M \longrightarrow 0
\] Each finitely-generated $R$-module has a minimal MCM approximation and a minimal FID hull. These sequences are unique up to isomorphism of exact sequences inducing the identity on $M$ \cite[Definitions 11.8 and 11.10, Proposition 11.13, Theorem 11.17]{LW12}. For an MCM approximation 
\[
0 \longrightarrow Y \xlongrightarrow{\iota} X \longrightarrow M \longrightarrow 0
\]
if $Y$ and $X$ have a nonzero direct summand $N$ in common via $\iota$, then $N$ is MCM and of finite injective dimension. It follows that $N \cong \omega^{(m)}$ for some positive integer $m$. Likewise, for an FID hull 
\[
0 \longrightarrow M \longrightarrow Y' \xlongrightarrow{\pi} X' \longrightarrow 0
\]
if $Y'$ and $X'$ have a nonzero direct summand $N'$ in common via $\pi$, then $N' \cong \omega^{(n)}$ for some positive integer $n$ \cite[Proposition 11.7]{LW12}. As a result, each MCM approximation and each FID hull of a finitely-generated $R$-module can be written as follows.\\
\begin{proposition}\cite[Propositions 1.5 and 1.6]{Ding90}
Let $(R,\mathfrak{m})$ be a Cohen-Macaulay local ring with canonical module $\omega$. Let $M$ be a finitely-generated $R$-module. Each MCM approximation of $M$ can be written as follows for some non-negative integer $m$.
\[
0 \longrightarrow \omega^{(m)} \oplus Y_{M} \longrightarrow \omega^{(m)} \oplus X_M \longrightarrow M \longrightarrow 0
\]\text{}\\ Likewise, each FID hull of $M$ can be written as follows for some non-negative integer $n$.
\[
0 \longrightarrow M \longrightarrow \omega^{(n)} \oplus Y^{M} \longrightarrow \omega^{(n)} \oplus X^{M} \longrightarrow 0
\] 
\end{proposition}
\text{}\\ A Cohen-Macaulay local ring $(R,\mathfrak{m})$ with canonical module $\omega$ is {\it{generically Gorenstein}} if $R_{\mathfrak{p}}$ is a Gorenstein local ring for each minimal prime ideal $\mathfrak{p}$ of $R$. When $R$ is not Gorenstein, this condition is equivalent to $\omega$ being isomorphic to a height one ideal of $R$ (see \cite[Proposition 3.3.18]{BH93} and \cite[Proposition 11.6]{LW12}). Such an ideal is called a {\it canonical ideal} and also denoted $\omega$.\\\\ In the following, I prove isomorphisms relating MCM approximations and FID hulls of modules that are constructed from canonical ideals. These results build on \cite[section 2]{Bartels26}. I use one of these relations to compute {\it{Auslander's $\delta$-invariant}} for $\left(\omega/xR \right)^{\vee}:=\text{Ext}_{R}^{1}\left(\omega/xR, \omega \right)$ and the {\it{$\gamma$-invariant}} for $\Omega_{R}^{1}\left(\omega/xR \right)$, the first syzygy of $\omega/xR$. Here, $x \in \omega$ is an $R$-regular element. \\
\begin{definition}\cite[Definition 11.24]{LW12}
Let $(R, \mathfrak{m})$ be a Cohen-Macaulay local ring with canonical module $\omega$. For a finitely-generated $R$-module $Z$, we define the {\it{free rank}} of $Z$, denoted $\text{f-rank}\, Z$, to be the rank of a maximal free direct summand of $Z$. In other words, 
\[
Z \cong \underline{Z} \oplus R^{(\text{f-rank}\, Z)},
\]
where $\underline{Z}$ has no non-trivial free direct summand. Dually, the {\it{canonical rank}} of $Z$, denoted $\omega$ \text{-rank} $Z$, is the largest integer $n$ such that $\omega^{(n)}$ is a direct summand of $Z$. For a finitely-generated $R$-module $M$, we define $\delta_{R}(M):=\text{f-rank}\, X_{M}$ and $\gamma_{R}(M):=\omega \text{-rank} \,X_{M}$.
\end{definition}
\text{}\\
In Corollary \ref{corollary:2.2} I prove that, for a canonical ideal $\omega \subset R$ and $x \in \omega$ an $R$-regular element, we have 
\[
\delta_{R}\left(\omega/xR \right)=\delta_{R}\left((\omega/xR)^{\vee} \right)=0
\]
and
\[\gamma_{R}\left(\Omega^{1}_{R}(\omega/xR) \right)=\gamma_{R}\left(\Omega^{1}_{R}\left((\omega/xR)^{\vee}\right) \right)=0.
\]
\text{}\\ The following lemma is used throughout section 2. \\ 
\begin{lemma}\cite[Lemma 2.7]{Bartels26}\label{lemma:1.3}
Let $(R,\mathfrak{m})$ be a Cohen-Macaulay local ring. If
\begin{equation*}\label{sequence:2.10.1}
0 \longrightarrow Y \longrightarrow M \longrightarrow X \longrightarrow 0
\end{equation*} is an exact sequence, $X$ is an MCM $R$-module, and $Y$ is an $R$-module of finite injective dimension, then the sequence splits and $M \cong Y \oplus X$.
\end{lemma}\text{}\\
\begin{center}
\section{Results}
\end{center} \text{}
\begin{definition} Let $(R,\mathfrak{m})$ be a Cohen-Macaulay local ring with canonical module $\omega$. We say that two $R$-modules $M$ and $N$ are $\omega$-{\it{stably isomorphic}}, and write $M \cong_{\omega} N$, if for some non-negative integers $s$ and $t$, we have 
\[
M \oplus \omega^{(s)} \cong N \oplus \omega^{(t)}.
\]
\end{definition}\text{}
\begin{remark}
Suppose $(R,\mathfrak{m})$ is a Cohen-Macaulay local ring with canonical module $\omega$ that is generically Gorenstein. If two $R$-modules $M$ and $N$ are $\omega$-stably isomorphic, then $M_{\mathfrak{p}}$ and $N_{\mathfrak{p}}$ are stably isomorphic $R_{\mathfrak{p}}$-modules for each minimal prime $\mathfrak{p}$ of $R$.
\end{remark}

\text{}\\In \cite{Bartels26}, I proved the following. 
\\
\begin{proposition}\cite[Proposition 2.10]{Bartels26}\label{prop:2.2}
Let $(R,\mathfrak{m})$ be a Cohen-Macaulay local ring that is generically Gorenstein and not Gorenstein. Let $\omega$ be a canonical ideal of $R$ and $x \in \omega$ an $R$-regular element. Let\, $\left(\omega/xR \right)^{\vee}:=\Ext_{R}^{1}\left(\omega/xR, \omega \right)$. Then we have the following.\newline
\begin{enumerate}[label=(\alph*)]
\item $\omega/xR$ is a Cohen-Macaulay $R$-module of codimension $1$ \newline 
\item $\left(\omega/xR \right)^{\vee}:=\Ext_{R}^{1}(\omega/xR, \omega)$ is a Cohen-Macaulay $R$-module of codimension $1$
\newline 
\item $X_{\omega/xR} \cong_{\omega} X_{\left(\omega/xR \right)^{\vee}} \cong_{\omega} X^{R/\omega}$ \newline 
\item There is an exact sequence 
\[
0 \longrightarrow R \longrightarrow \omega^{(n)} \longrightarrow X^{R/\omega} \longrightarrow 0\] with\, $n=\mu_{R}(\omega)$.
\end{enumerate}
\end{proposition}
\text{}\\ In the following proposition, I further describe minimal MCM approximations and FID hulls of the modules $\omega/xR$, $\left(\omega/xR\right)^{\vee}$, and $R/\omega$.\\\\
\begin{proposition}
Let $(R,\mathfrak{m})$ be a Cohen-Macaulay local ring with canonical module that is generically Gorenstein and not Gorenstein. Let $\omega$ be a canonical ideal of $R$ and $x \in \omega$ an $R$-regular element. Then exactly one of the following is true.\\ 
\begin{enumerate}
\item $X_{\omega/xR} \cong X^{R/\omega} \cong X^R$ \\
\item $X_{\omega/xR} \cong X^{R/\omega} \cong \omega \oplus X^R$ \\
\item $X_{\omega/xR} \cong \omega \oplus X^{R/\omega} \cong \omega \oplus X^R$
\end{enumerate}\text{}\\\\ Moreover, for\, $(\omega/xR)^{\vee}:={\Ext}^{1}_{R}(\omega/xR,\omega)$, exactly one of the following is true. \\
\begin{enumerate}[start=4]
 \item $R \oplus X_{(\omega/xR)^{\vee}} \cong X_{R/xR}$ \\
 \item $R \oplus X_{(\omega/xR)^{\vee}} \cong \omega \oplus X_{R/xR}$
\end{enumerate}  
\end{proposition}\text{}
\begin{proof} Consider the pullback diagram for the exact sequence \\
\begin{equation}
0 \longrightarrow xR/x\omega \longrightarrow \omega/x\omega \longrightarrow \omega/xR \longrightarrow 0
\end{equation}\text{}\\ and the minimal MCM approximation of $\omega/xR$.

\begin{equation}\label{diagram:1.2}
\xymatrix{&& 0 \ar[d] & 0 \ar[d] \\
&& Y_{\omega/xR} \ar@{=}[r] \ar[d] & Y_{\omega/xR} \ar[d]\\  
0 \ar[r] &
xR/x\omega \ar@{=}[d] \,\ar[r] & Z \ar[r] \ar[d] & X_{\omega/xR} \ar[d] \ar[r]  & 0 \\
0 \ar[r] & xR/x\omega \ar[r] & \omega/x\omega \ar[r] \ar[d] & \omega/xR \ar[r] \ar[d]& 0 \\
&& 0 & 0 \\
} 
\end{equation}
\text{}\\
From the middle column of diagram \ref{diagram:1.2}, we see that $Z$ has finite injective dimension. Therefore, the middle row is an FID hull for $R/\omega$\, and\, $X_{\omega/xR} \cong \omega^{(s)} \oplus X^{R/\omega}$ for some $s \geq 0$. By the exact sequence in Proposition \ref{prop:2.2} (d), we have $X^{R/\omega} \cong \omega^{(t)} \oplus X^R$ for some $t \geq 0$. Finally, by \cite[Proposition 2.9 (c)]{Bartels26}, we have $X_{\omega/xR} \cong \omega^{(r)} \oplus X^R$, where $r=0$ or $r=1$. First suppose\, $r=0$. Then we have 
\[
X_{\omega/xR} \cong \omega^{(s+t)} \oplus X_{\omega/xR}.
\] So \,$s=t=0$ and 
\[
X_{\omega/xR} \cong X^{R/\omega} \cong X^R.
\] \text{}\\ Now suppose \,$r=1$. Then we have
\[
\omega \oplus X^R \cong \omega^{(s+t)} \oplus X^R.
\]\text{}\\ So\, $s=0$ \,and\, $t=1$ \,or\, $s=1$ \,and\, $t=0$. If\, $s=0$  and  $t=1$, \, then we have \\
\[
X_{\omega/xR} \cong X^{R/\omega} \cong \omega \oplus X^R.
\] If $s=1$ and $t=0$, then we have 
\[
X_{\omega/xR} \cong \omega \oplus X^{R/\omega} \cong \omega \oplus X^{R}.
\] Consider again the sequence
\[
0 \longrightarrow xR/x\omega \longrightarrow \omega/x\omega \longrightarrow \omega/xR \longrightarrow 0.
\]\text{}\\ Each module in this sequence is a Cohen-Macaulay $R$-module of codimension $1$ by Proposition \ref{prop:2.2} (a) and \cite[Proposition 3.3.18]{BH93}. Applying $\text{Hom}_{R}(-,\omega)$, we obtain the exact sequence
\begin{equation}\label{sequence 1.3}
0 \longrightarrow (\omega/xR)^{\vee} \longrightarrow R/xR \longrightarrow R/\omega \longrightarrow 0.
\end{equation} \text{}\newline Consider the pullback diagram for sequence \ref{sequence 1.3} and the exact sequence 
\[
0 \longrightarrow \omega \longrightarrow R \longrightarrow R/\omega \longrightarrow 0.
\] 
\begin{equation}\label{diagram:1.3}
\xymatrix{&& 0 \ar[d] & 0 \ar[d] \\
&& \omega \ar@{=}[r] \ar[d] & \omega \ar[d]\\  
0 \ar[r] &
(\omega/xR)^{\vee} \ar@{=}[d] \,\ar[r] & Z \ar[r] \ar[d] & R \ar[d] \ar[r]  & 0 \\
0 \ar[r] & (\omega/xR)^{\vee} \ar[r] & R/xR \ar[r] \ar[d] & R/\omega \ar[r] \ar[d]& 0 \\
&& 0 & 0 \\
} 
\end{equation}
\text{}\vspace{0.3cm}
The middle row of diagram \ref{diagram:1.3} splits, so the middle column gives us the exact sequence 
\begin{equation}\label{sequence:1.5}
0 \longrightarrow \omega \longrightarrow R \oplus (\omega/xR)^{\vee} \longrightarrow R/xR \longrightarrow 0.
\end{equation}
\text{}\\
Now consider the pullback diagram for sequence \ref{sequence:1.5} and the minimal MCM approximation of $R/xR$. 

\begin{equation}\label{diagram:1.6}
\xymatrix{&& 0 \ar[d] & 0 \ar[d] \\
&& Y_{R/xR} \ar@{=}[r] \ar[d] & Y_{R/xR} \ar[d]\\  
0 \ar[r] &
\omega \ar@{=}[d] \,\ar[r] & Z \ar[r] \ar[d] & X_{R/xR} \ar[d] \ar[r]  & 0 \\
0 \ar[r] & \omega \ar[r] & R \oplus (\omega/x\omega)^{\vee} \ar[r] \ar[d] & R/xR \ar[r] \ar[d]& 0 \\
&& 0 & 0 \\
}
\end{equation}
\text{}\\ The middle row of diagram \ref{diagram:1.6} splits by Lemma \ref{lemma:1.3}, so the middle column gives us the exact sequence 
\begin{equation}\label{sequence:1.7}
0 \longrightarrow Y_{R/xR} \longrightarrow \omega \oplus X_{R/xR} \longrightarrow R \oplus (\omega/x\omega)^{\vee} \longrightarrow 0.
\end{equation} \text{}\\ Sequence \ref{sequence:1.7} is an MCM approximation of $R \oplus (\omega/x\omega)^{\vee}$, so we have \\
\[
\omega \oplus X_{R/xR} \cong_{\omega} X_{R \,\oplus\, (\omega/xR)^{\vee}} \cong R \oplus X_{(\omega/xR)^{\vee}}
\]\text{}\\ and for some $a \geq 0$, we have
\begin{equation}\label{isom:1.8}
\omega \oplus X_{R/xR}  \cong \omega^{(a)} \oplus R \oplus X_{(\omega/xR)^{\vee}}.
\end{equation} \text{}\\ Now consider the pullback diagram for sequence \ref{sequence:1.5} and the exact sequence 
\[
0 \longrightarrow Y_{(\omega/xR)^{\vee}} \longrightarrow R\oplus X_{(\omega/xR)^{\vee}} \longrightarrow R \oplus (\omega/xR)^{\vee} \longrightarrow 0. 
\]

\begin{equation}\label{diagram:2.1}
\xymatrix{& 0 \ar[d] & 0 \ar[d] \\
& Y_{(\omega/xR)^{\vee}} \ar@{=}[r] \ar[d] & Y_{(\omega/xR)^{\vee}} \ar[d]\\  
0 \ar[r] &
Z \ar[d] \,\ar[r] & R \oplus X_{(\omega/xR)^{\vee}} \ar[r] \ar[d] & R/xR \ar@{=}[d] \ar[r]  & 0 \\
0 \ar[r] & \omega \ar[r] \ar[d] & R \oplus (\omega/xR)^{\vee} \ar[r] \ar[d] & R/xR \ar[r] & 0 \\
& 0 & 0 \\
}
\end{equation}
\text{}\\ By the first column, we see that $Z$ has finite injective dimension. Therefore, The middle row is an MCM approximation of $R/xR$, and for some $b \geq 0$, we have \\
\[
R \oplus X_{(\omega/xR)^{\vee}} \cong \omega^{(b)} \oplus X_{R/xR}.
\]\text{}\\
With isomorphism \ref{isom:1.8}, this gives us \\
\[
\omega \oplus X_{R/xR} \cong \omega^{(a+b)} \oplus X_{R/xR}.
\]\text{}\\
Therefore, we have $a=0$ \,and\, $b=1$\, or \,$a=1$\, and\, $b=0$. If \,$a=0$  and  $b=1$,  then we have \\
\begin{equation}
R \oplus X_{(\omega/xR)^{\vee}} \cong \omega \oplus X_{R/xR}.
\end{equation} \text{}\\ If $a=1$ and $b=0$, \,then 
\[
R \oplus X_{(\omega/xR)^{\vee}} \cong X_{R/xR}.
\] \end{proof} \text{}\\
\begin{corollary}\label{corollary:2.2}
Let $(R,\mathfrak{m})$ be a Cohen-Macaulay local ring with canonical module that is generically Gorenstein and not Gorenstein. Let \,$\omega$ be a canonical ideal of $R$ and $x \in \omega$ an $R$-regular element. Let \,$(\omega/xR)^{\vee}:={\Ext}^{1}_{R}(\omega/xR,\omega)$. Then we have the following.\\
\begin{enumerate}
\item \,$\delta_{R}\left(\omega/xR \right)=\delta_{R}\left((\omega/xR)^{\vee} \right)=0$
\\\\
\item $\gamma_{R}\left(\Omega^{1}_{R}(\omega/xR) \right)=\gamma_{R}\left(\Omega^{1}_{R}\left((\omega/xR)^{\vee}\right) \right)=0$
\end{enumerate}
\end{corollary}\text{}\\
\begin{proof} Since $x \in \omega$, we have the surjection $R/xR \longrightarrow R/\omega$. By 
\cite[Corollary 11.28 (ii) and (iii)]{LW12}, we have 
\[
\delta_{R}(R/\omega) \leq \delta_{R}(R/xR) \leq \mu_{R}(R/xR).
\]
Since the exact sequence
\[
0 \longrightarrow \omega \longrightarrow R \longrightarrow R/\omega \longrightarrow 0
\] \text{}\\ is the minimal MCM approximation of $R/\omega$, we have $\delta_{R}(R/xR)=1$. By the above proposition, we have  
\[
R \oplus X_{(\omega/xR)^{\vee}} \cong X_{R/xR}
\]
or
\[R \oplus X_{(\omega/xR)^{\vee}} \cong \omega \oplus X_{R/xR}.
\]
\text{}\\ Since $R$ is not Gorenstein, $\omega$ does not have a non-trivial free direct summand. If\, $\delta_{R}\left((\omega/xR)^{\vee} \right)>0$, then $X_{R/xR}$ has a free summand of rank two, giving us $\delta_{R}\left(R/xR\right)>1$. The other equality in statement (1) follows from the surjection $\omega \longrightarrow \omega/xR$ and \cite[Corollary 11.28 (ii)]{LW12}. Statement (2) follows from statement (1), Proposition \ref{prop:2.2} (a) and (b), and \cite[Theorem 11.5 \,and\, Proposition 11.35]{LW12}. 
\end{proof}\text{}\\
\begin{proposition}
Let $(R,\mathfrak{m})$ be a Cohen-Macaulay local ring with canonical module that is generically Gorenstein and not Gorenstein. Let $\omega$ be a canonical ideal of $R$ and $x \in \omega$ an $R$-regular element. Let \,$\left(-\right)^{\vee}:=\Ext_{R}^{1}\left(-,\omega \right)$. Then \\
\begin{enumerate}
\item $X^{\omega/x^mR} \cong_{\omega} X^{\omega/x^nR}$\text{}\\\\ for all\, $m,n \geq 1$. Moreover, for sufficiently large $N \geq 1$, we have \\
\[
X^{\omega/x^mR} \cong X^{\omega/x^nR}
\] \text{}for all $m,n \geq N$. \\\\\\
\item $X^{(\omega/x^mR)^{\vee}} \cong_{\omega} X^{(\omega/x^nR)^{\vee}}$ \text{}\\\\ for all\, $m,n \geq 1$. Moreover, for sufficiently large $N \geq 1$, we have \\
\[
X^{(\omega/x^mR)^{\vee}} \cong X^{(\omega/x^nR)^{\vee}}
\] for all $m,n \geq N$.
\end{enumerate}
\end{proposition} \text{}\\
\begin{proof}
Let $n \geq 1$. Consider the pushout diagram for the exact sequence \\
\begin{equation}\label{sequence:1.11}
0 \longrightarrow x\omega/x^{n+1}R \longrightarrow \omega/x^{n+1}R \longrightarrow \omega/x\omega \longrightarrow 0.
\end{equation} \text{}\\ and the minimal FID hull of $\omega/x^nR$. \\\\
\begin{equation}\label{diagram:1.12}
\xymatrix{&& 0 \ar[d] & 0 \ar[d] \\
&0 \ar[r]& x\omega/x^{n+1}R \ar[r] \ar[d] & \omega/x^{n+1} R \ar[r] \ar[d] & \omega/x \omega \ar@{=}[d] \ar[r] & 0\\  
&0 \ar[r] &
Y^{\omega/x^nR} \,\ar[r] \ar[d] & Z \ar[r] \ar[d] & \omega/x\omega \ar[r]  & 0 \\
&& X^{\omega/x^nR} \ar@{=}[r] \ar[d] & X^{\omega/x^nR} \ar[d] \\
&& 0 & 0 \\
}
\end{equation}
\\
Since $Y^{\omega/x^nR}$ and $\omega/x\omega$ have finite injective dimension, $Z$ also has finite injective dimension. So the middle column of diagram \ref{diagram:1.12} is an FID hull for $\omega/x^{n+1}R$ \,and \\
\[
X^{\omega/x^nR} \cong \omega^{(s)} \oplus X^{\omega/x^{n+1}R}
\] for some $s \geq 0$. \\\\
Suppose it is not true that there exists a positive integer $N$ such that \\
\[
X^{\omega/x^mR} \cong X^{\omega/x^{n}R}
\] \text{}\\ for all $m, n \geq N$. Then there is an infinite, strictly increasing sequence of positive integers $n_1<n_2<\cdots $ such that 
\[
X^{\omega/x^{n_i}R} \cong \omega^{(s_i)} \oplus X^{\omega/x^{n_{i+1}}R}
\] \text{}\\ where $s_i >0$ for all $i \geq 1$. Therefore, for every positive integer $\kappa$, we have \\
\[
X^{\omega/x^{n_1}R} \cong \left(\bigoplus\limits_{i=1}^{\kappa} \omega^{(s_i)} \right) \oplus X^{\omega/x^{n_{\kappa+1}}R}
\]\text{}\\\\ so that \,$\mu_{R}\left(X^{\omega/x^{n_1}R} \right) \geq \kappa$. This contradicts that $X^{\omega/x^{n_1}R}$ is finitely-generated. \\\\
Now dualize the sequence
\[
0 \longrightarrow x^nR/x^{n+1}R \longrightarrow \omega/x^{n+1}R \longrightarrow \omega/x^{n}R \longrightarrow 0
\]\text{}
to obtain the exact sequence \\
\begin{equation}\label{sequence:1.13}
0 \longrightarrow (\omega/x^{n}R)^{\vee} \longrightarrow (\omega/x^{n+1}R)^{\vee} \longrightarrow \omega/x\omega \longrightarrow 0.
\end{equation} \text{}\\ As in diagram \ref{diagram:1.12}, we construct the pushout diagram for sequence \ref{sequence:1.13} and the minimal FID hull for $(\omega/x^nR)^{\vee}$ to obtain the isomorphism \\
\[
X^{(\omega/x^nR)^{\vee}} \cong \omega^{(t)} \oplus X^{(\omega/x^{n+1}R)^{\vee}}
\] \text{}\\ for some $t \geq 0$. Applying the above argument, it follows that, for some positive integer $N$, we have 
\[
X^{(\omega/x^mR)^{\vee}} \cong X^{(\omega/x^{n}R)^{\vee}}
\]
for all $m, n \geq N$. 
\end{proof}\text{}\\

\begin{proposition}\label{prop:2.7}
Let $(R,\mathfrak{m})$ be a Cohen-Macaulay local ring with canonical module that is generically Gorenstein. Let $\omega$ be a canonical ideal of $R$ and $x \in \omega$ an $R$-regular element. Then exactly one of the following is true.\\ 
\begin{enumerate}
\item $X_{\omega/\omega^2} \cong X^{\omega^2/x \omega}$ \\
\item $X_{\omega/\omega^2} \cong \omega \oplus X^{\omega^2/x\omega}$
\end{enumerate} 
\end{proposition}\text{}\\
\begin{proof}
Consider the pullback diagram for the exact sequence\\ 
\begin{equation}\label{sequence:1.19}
0 \longrightarrow \omega^2/x\omega \longrightarrow \omega/x\omega \longrightarrow \omega/\omega^2 \longrightarrow 0
\end{equation}\text{}\\ and the minimal MCM approximation of\, $\omega/\omega^2$.

\begin{equation}
\label{diagram:1.20}
\xymatrix{&& 0 \ar[d] & 0 \ar[d] \\
&& Y_{\omega/\omega^2} \ar@{=}[r] \ar[d] & Y_{\omega/\omega^2} \ar[d]\\  
0 \ar[r] &
\omega^2/x\omega \ar@{=}[d] \,\ar[r] & Z \ar[r] \ar[d] & X_{\omega/\omega^2} \ar[d] \ar[r]  & 0 \\
0 \ar[r] & \omega^2/x\omega \ar[r] & \omega/x\omega \ar[r] \ar[d] & \omega/\omega^2 \ar[r] \ar[d]& 0 \\
&& 0 & 0 \\
}
\end{equation}
\text{}\\ Since $Y_{\omega/\omega^2}$ and $\omega/x\omega$ have finite injective dimension, $Z$ also has finite injective dimension. Therefore, the middle row is an FID hull for $\omega^2/x\omega$ and \\
\begin{equation}\label{isom:1.21}
X_{\omega/\omega^2} \cong \omega^{(s)} \oplus X^{\omega^2/x\omega}.
\end{equation} \text{}\\ for some $s \geq 0$. Now consider the pushout diagram for sequence \ref{sequence:1.19} and the minimal FID hull of $\omega^2/x\omega$.

\begin{equation}\label{diagram:2.17}
\xymatrix{&& 0 \ar[d] & 0 \ar[d] \\
&0 \ar[r]& \omega^2/x\omega \ar[r] \ar[d] & \omega/x\omega \ar[r] \ar[d] & \omega/\omega^2 \ar@{=}[d] \ar[r] & 0\\  
&0 \ar[r] &
Y^{\omega^2/x\omega} \,\ar[r] \ar[d] & Z \ar[r] \ar[d] & \omega/\omega^2 \ar[r]  & 0 \\
&& X^{\omega^2/x\omega} \ar@{=}[r] \ar[d] & X^{\omega^2/x\omega} \ar[d] \\
&& 0 & 0 \\
}
\end{equation}
\text{}\\ 
Since $\omega/x\omega$ has finite injective dimension and $X^{\omega^2/x\omega}$ is an MCM $R$-module, the middle column of diagram \ref{diagram:2.17} splits and we obtain the exact sequence \\
\begin{equation}\label{sequence:2.18}
0 \longrightarrow Y^{\omega^2/x\omega} \longrightarrow X^{\omega^2/x\omega} \oplus \omega/x\omega \longrightarrow \omega/\omega^2 \longrightarrow 0.
\end{equation}
\text{}\\ 
Consider the pullback diagram for sequence \ref{sequence:2.18} and the sequence \\
\[
0 \longrightarrow x\omega \longrightarrow X^{\omega^2/x\omega} \oplus \omega \longrightarrow X^{\omega^2/x\omega} \oplus \omega/x\omega \longrightarrow 0.
\]

\begin{equation}\label{diagram:2.19}
\xymatrix{& 0 \ar[d] & 0 \ar[d] \\
& x\omega \ar@{=}[r] \ar[d] & x\omega \ar[d]\\  
0 \ar[r] &
Z \ar[d] \,\ar[r] & X^{\omega^2/x\omega} \oplus \omega \ar[r] \ar[d] & \omega/\omega^2 \ar@{=}[d] \ar[r]  & 0 \\
0 \ar[r] & Y^{\omega^2/x\omega} \ar[r] \ar[d] & X^{\omega^2/x\omega} \oplus \omega/x\omega \ar[r] \ar[d] & \omega/\omega^2 \ar[r] & 0 \\
& 0 & 0 \\
}
\end{equation}
\text{}\\ Since $x\omega$ and $Y^{\omega^2/x\omega}$ have finite injective dimension, $Z$ also has finite injective dimension. Therefore, the middle row of diagram \ref{diagram:2.19} is an MCM approximation of $\omega/\omega^2$ and \\
\begin{equation}
\omega \oplus X^{\omega^2/x\omega} \cong  \omega^{(t)} \oplus X_{\omega/\omega^2}
\end{equation} \text{}\\ for some $t \geq 0$. By isomorphism \ref{isom:1.21}, we have \\
\begin{equation}
\omega \oplus X^{\omega^2/x\omega} \cong \omega^{(s+t)} \oplus X^{\omega^2/x\omega}
\end{equation}
\text{}\\ It follows that $s+t=1$.  So \, $s=0$  and  $t=1$ \, or \, $s=1$  and  $t=0$. \\\\ If $s=0$ and $t=1$, then 
\[
X_{\omega/\omega^2} \cong X^{\omega^2/x\omega}.
\]
If $s=1$ and $t=0$, then 
\[
X_{\omega/\omega^2} \cong \omega \oplus X^{\omega^2/x\omega}.
\]
\end{proof}\text{}
\begin{corollary}
Exactly one of the following is true.\\
\begin{enumerate}
\item $X_{\omega/\omega^2} \cong X^{\omega^2/x\omega} \cong X^{\omega^2}$
\\\\
\item $X_{\omega/\omega^2} \cong \omega \oplus X^{\omega^2/x\omega} \cong X^{\omega^2}$ \\\\
\item $X_{\omega/\omega^2} \cong \omega \oplus X^{\omega^2/x\omega} \cong \omega \oplus X^{\omega^2}$
\end{enumerate}
\end{corollary}
\text{}
\begin{proof}
This follows from Proposition \ref{prop:2.7}, \cite[Proposition 2.9 (c)]{Bartels26}, and the pushout diagram for the sequence 
\[
0 \longrightarrow x\omega \longrightarrow \omega^2 \longrightarrow \omega^2/x\omega \longrightarrow 0
\] and the minimal FID hull for $\omega^2$.
\end{proof}
\text{}\\
\begin{proposition}
Let $(R,\mathfrak{m})$ be a Cohen-Macaulay local ring with canonical module that is generically Gorenstein. Let $\omega$ be a canonical ideal of $R$ and $x \in \omega$ an $R$-regular element. Then exactly one of the following is true.\\ 
\begin{enumerate}
\item $X_{\omega^2} \cong X_{\omega^2/x \omega}$ \,\, and \,\, $\omega \oplus Y_{\omega^2} \cong Y_{\omega^2/x\omega}$ \\\\
\item $X_{\omega^2} \cong \omega \oplus X_{\omega^2/x\omega}$ \,\, and \,\, $Y_{\omega^2} \cong Y_{\omega^2/x\omega}$
\end{enumerate} 
\end{proposition}\text{}
\begin{proof}
Follows from \cite[Proposition 2.8 (c)]{Bartels26}.
\end{proof}
\text{}\\
\begin{proposition}
Let $(R,\mathfrak{m})$ be a Cohen-Macaulay local ring with canonical module that is generically Gorenstein. Let $\omega$ be a canonical ideal of $R$. Then exactly one of the following is true. \\
\begin{enumerate}
\item $R \oplus X_{\omega/\omega^2} \cong X_{R/\omega^2}$ \,\, and \,\, $\omega \oplus Y_{\omega/\omega^2} \cong Y_{R/\omega^2}$\\\\
\item $R \oplus X_{\omega/\omega^2} \cong \omega \oplus X_{R/\omega^2}$ \,\, and \,\, $ Y_{\omega/\omega^2} \cong Y_{R/\omega^2}$
\end{enumerate}
\end{proposition}
\text{}
\begin{proof} Consider the pullback diagram for the exact sequences 
\[
0 \longrightarrow \omega \longrightarrow R \longrightarrow R/\omega \longrightarrow 0
\] and 
\[
0 \longrightarrow \omega/\omega^2 \longrightarrow R/\omega^2 \longrightarrow R/\omega \longrightarrow 0.
\]

\begin{equation}\label{diagram:1.22}
\xymatrix{&& 0 \ar[d] & 0 \ar[d] \\
&& \omega \ar@{=}[r] \ar[d] & \omega \ar[d]\\  
0 \ar[r] &
\omega/\omega^2 \ar@{=}[d] \,\ar[r] & Z \ar[r] \ar[d] & R \ar[d] \ar[r]  & 0 \\
0 \ar[r] & \omega/\omega^2 \ar[r] & R/\omega^2 \ar[r] \ar[d] & R/\omega \ar[r] \ar[d]& 0 \\
&& 0 & 0 \\
} 
\end{equation}
\text{}\\
The middle row of diagram \ref{diagram:1.22} splits, so we obtain the exact sequence\\
\begin{equation}\label{sequence:1.23}
0 \longrightarrow \omega \longrightarrow R \oplus \omega/\omega^2 \longrightarrow R/\omega^2 \longrightarrow 0.
\end{equation}\text{}\\
Consider the pullback diagram for sequence \ref{sequence:1.23} and the sequence\\
\[
0 \longrightarrow Y_{\omega/\omega^2} \longrightarrow R \oplus X_{\omega/\omega^2} \longrightarrow R \oplus \omega/\omega^2 \longrightarrow 0.
\]

\begin{equation}\label{diagram:1.24}
\xymatrix{& 0 \ar[d] & 0 \ar[d] \\
& Y_{\omega/\omega^2} \ar@{=}[r] \ar[d] & Y_{\omega/\omega^2} \ar[d]\\  
0 \ar[r] &
Z \ar[d] \,\ar[r] & R \oplus X_{\omega^2/x\omega} \ar[r] \ar[d] & R/\omega^2 \ar@{=}[d] \ar[r]  & 0 \\
0 \ar[r] & \omega \ar[r] \ar[d] & R \oplus \omega/\omega^2 \ar[r] \ar[d] & R/\omega^2 \ar[r] & 0 \\
& 0 & 0 \\
}
\end{equation}
\text{}\\ Since $Y_{\omega/\omega^2}$ has finite injective dimension and $\omega$ is an MCM $R$-module, the first column of diagram \ref{diagram:1.24} splits, and the middle row is an MCM approximation of $R/\omega^2$. Therefore, we have \\
\[
\omega \oplus Y_{\omega/\omega^2} \cong \omega^{(s)} \oplus Y_{R/\omega^2}
\] and 
\[
R \oplus X_{\omega/\omega^2} \cong \omega^{(s)} \oplus X_{R/\omega^2}
\]\text{}\\
for some $s \geq 0$. Now consider the pullback diagram for sequence \ref{sequence:1.23} and the minimal MCM approximation of $R/\omega^2$.

\begin{equation}\label{diagram:1.22}
\xymatrix{&& 0 \ar[d] & 0 \ar[d] \\
&& Y_{R/\omega^2} \ar@{=}[r] \ar[d] & Y_{R/\omega^2} \ar[d]\\  
0 \ar[r] &
\omega \ar@{=}[d] \,\ar[r] & Z \ar[r] \ar[d] & X_{R/\omega^2} \ar[d] \ar[r]  & 0 \\
0 \ar[r] & \omega \ar[r] & R \oplus \omega/\omega^2 \ar[r] \ar[d] & R/\omega^2 \ar[r] \ar[d]& 0 \\
&& 0 & 0 \\
} 
\end{equation}
\text{}\\ Since $\omega$ has finite injective dimension and $X_{R/\omega^2}$ is MCM, the middle row splits and the middle column gives us the exact sequence \\
\begin{equation}\label{sequence:1.26}
0 \longrightarrow Y_{R/\omega^2} \longrightarrow \omega \oplus X_{R/\omega^2} \longrightarrow R \oplus \omega/\omega^2 \longrightarrow 0.
\end{equation}
\text{}\\ Sequence \ref{sequence:1.26} is an MCM approximation of $R \oplus \omega/\omega^2$. Therefore, we have \\
\[
Y_{R/\omega^2} \cong \omega^{(t)} \oplus Y_{R \,\oplus\, \omega/\omega^2 } \cong \omega^{(t)} \oplus Y_{\omega/\omega^2}
\]
and 
\[
\omega \oplus X_{R/\omega^2} \cong \omega^{(t)} \oplus X_{\omega/\omega^2\, \oplus\, R} \cong \omega^{(t)} \oplus R \oplus X_{\omega/\omega^2}.
\]
\text{}\\ for some $t \geq 0$. Therefore, 
\[
\omega \oplus X_{R/\omega^2} \cong \omega^{(s+t)} \oplus X_{R/\omega^2}
\]
\text{}\\ and \,$s+t=1$. So\, $s=0$ and $t=1$ \,or\, $s=1$ and $t=0$.
\text{}\\\\ If $s=0$ and $t=1$, then \\
\begin{equation*}\label{isom:1.27}
R \oplus X_{\omega/\omega^2} \cong X_{R/\omega^2} \,\,\,\, \text{and} \,\,\,\, \omega \oplus Y_{\omega/\omega^2} \cong Y_{R/\omega^2}.
\end{equation*}\text{}\\ If $s=1$ and $t=0$, then \\
\begin{equation*}\label{isom:1.28}
R \oplus X_{\omega/\omega^2} \cong \omega \oplus X_{R/\omega^2} \,\,\,\, \text{and} \,\,\,\, Y_{\omega/\omega^2} \cong Y_{R/\omega^2}.
\end{equation*} \end{proof} 

\text{}\newline

\bibliographystyle{amsplain}
}
\end{document}